\documentclass[a4paper,12pt]{article}

\usepackage[pdftex]{graphicx}
\usepackage[section] {placeins}
\usepackage{enumerate}
\usepackage[normalem]{ulem}
\usepackage{amsmath}
\usepackage{amsthm}
\usepackage{latexsym}
\usepackage{amsfonts}
\usepackage{amssymb}
\usepackage{mathtools}
\usepackage{longtable}

\usepackage{tikz-cd}
\usetikzlibrary{graphs}
\usetikzlibrary{arrows.meta}
\usetikzlibrary{shapes.geometric}

\usepackage{authblk}

\usepackage{latexsym,amssymb,amsfonts, amsthm, amsmath}
\usepackage[english]{babel}
\usepackage{bm}
\usepackage{mathrsfs}
\usepackage{algpseudocode}
\usepackage{algorithm}

\algnewcommand{\IIf}[1]{\State\algorithmicif\ #1\ \algorithmicthen}
\algnewcommand{\EndIIf}{\unskip\ \algorithmicend\ \algorithmicif}

\newtheorem{thm}{Theorem}[section]
\newtheorem{lem}[thm]{Lemma}

\newtheorem{prop}[thm]{Proposition}
\newtheorem{conj}[thm]{Conjecture}

\newtheorem{rem}[thm]{Remark}

\theoremstyle{definition}
\newtheorem{defn}[thm]{Definition}

\newcommand{\Z}{\mathbb{Z}}

\usepackage{fullpage}

\makeatletter
\newcommand{\subjclass}[2][1991]{%
  \let\@oldtitle\@title%
  \gdef\@title{\@oldtitle\footnotetext{#1 \emph{Mathematics subject classification.} #2}}%
}
\newcommand{\keywords}[1]{%
  \let\@@oldtitle\@title%
  \gdef\@title{\@@oldtitle\footnotetext{\emph{Key words and phrases.} #1.}}%
}
\makeatother

\title{Weak Sequenceability in Cyclic Groups}

\date{}
\author[1]{Simone Costa}
\author[2]{Stefano Della Fiore}

\affil[1]{DICATAM, Sez.~Matematica, Universit\`a degli Studi di Brescia, \newline Via Branze~43, I~25123 Brescia, Italy}
\affil[2]{DII, Universit\`a degli Studi di Brescia, Via Branze~38, \newline I~25123 Brescia, Italy}

\begin{document}
\maketitle
\begin{abstract}
\noindent A subset $A$ of an abelian group $G$ is {\em sequenceable} if there is an ordering $(a_1, \ldots, a_k)$ of its elements such that the partial sums~$(s_0, s_1, \ldots, s_k)$, given by $s_0 = 0$ and $s_i = \sum_{j=1}^i a_i$ for $1 \leq i \leq k$, are distinct, with the possible exception that we may have~$s_k = s_0 = 0$. In the literature there are several conjectures and questions concerning the sequenceability of subsets of abelian groups, which have been combined and summarized in \cite{AL20} into the conjecture that if a subset of an abelian group does not contain~0 then it is sequenceable.
If the elements of a sequenceable set $A$ do not sum to $0$ then there exists a simple path $P$ in the Cayley graph $Cay[G:\pm A]$ such that $\Delta(P) = \pm A$.

In this paper, inspired by this graph-theoretical interpretation, we propose a weakening of this conjecture. Here, under the above assumptions,  we want to find an ordering whose partial sums define a walk $W$ of girth bigger than $t$ (for a given $t < k$) and such that $\Delta(W) = \pm A$.  This is possible given that the partial sums $s_i$ and $s_j$ are different whenever $i$ and $j$ are distinct and $|i-j|\leq t$. 
In this case, we say that the set $A$ is $t$-{\em weak sequenceable}. The main result here presented is that any subset $A$ of $\mathbb{Z}_p\setminus \{0\}$ is $t$-weak sequenceable whenever $t<7$ or when $A$ does not contain pairs of type $\{x,-x\}$ and $t<8$.
\end{abstract}

\section{Introduction}\label{sec:intro}
One of the most important tools in the construction of combinatorial designs of various kinds is given by the difference method and its variations. The research of new and efficient ways to exploit these methods leads to several interesting conjectures proposed by many authors such as Archdeacon and Graham (and many others, see \cite{ADMS16,BH05,CDOR,CMPP18,Graham71,HR09}). Results on these conjectures have several implications in the constructions of combinatorial designs, graph decompositions and have applications to other combinatorial structures such as the Heffer arrays (see, for instance, \cite{A15}). 
The underlying problem is to provide an ordering to the elements of a given subset $A$ (or multiset) of an abelian group $G$ so that the partial sums define a simple path $P$ or a simple cycle $C$ in the Cayley graph $Cay[G:\pm A]$ (the graph whose vertex set is $G$ and whose edges are the pairs $\{x,y\}$ such that $x-y\in \pm A$) if the sum of the elements is zero. The connection with the difference method is that, here, the elements of $A$ are the differences between adjacent vertices of $P$.

We introduce some definitions and notation (we adopt those of \cite{CDOR}) to make these questions more precise. Let~$G$ be an abelian group and let $A$ be a subset of $G \setminus \{ 0 \}$ whose size is equal to $k$. Let ${\bm \omega} = (a_1, a_2, \ldots, a_k)$ be an ordering of the elements of~$A$ and we define its partial sums ${\bm s} = (s_0, s_1, \ldots, s_k)$ by $s_0=0$ and $s_i = a_1 + \cdots + a_i$ for $i>0$.  We denote the sum of the elements of~$A$ by $\Sigma A$. As $G$ is abelian, for any ordering of the elements of~$A$ the final partial sum $s_k$ is equal to $\Sigma A$.
Then, as done in \cite{AL20} and \cite{CDOR}, we give the following definitions.
\begin{defn}
\begin{itemize}
\item The ordering ${\bm \omega}$ is said to be a {\em sequencing} (or a {\em linear sequencing}) of~$A$ if the elements of ${\bm s}$ are distinct;
\item the ordering ${\bm \omega}$ is said to be an {\em R-sequencing} (or a {\em rotational sequencing}) of~$A$ if the elements of ${\bm s}$ are distinct with the exception that $s_0 = 0 = s_k$;
\item a subset $A$ of an abelian group is said to be {\em sequenceable} if it admits a sequencing or an R-sequencing;
\item an abelian group $G$ is said to be {\em strongly sequenceable} if every subset $A$ of $G\setminus \{0\}$ is sequenceable.
\end{itemize}
\end{defn}
We remark that a set $A$ can have a linear sequencing only if $\Sigma A\not=0$. On the other hand, $A$ can have a rotational sequencing only when $\Sigma A=0$.

We can now state the main conjecture about sequenceability, first suggested with a different terminology in \cite{AL20} (see Conjecture 3.4 therein), that is the amalgamation of several questions and conjectures (see also \cite{ADMS16,BH05,CDOR,CMPP18,Graham71}).
\begin{conj}[Alspach-Liversidge \cite{AL20}]\label{conj:main}
Every abelian group is strongly sequenceable.
\end{conj}

Theorem~\ref{th:known} summarizes the main known results concerning Conjecture~\ref{conj:main}.
\begin{thm}\label{th:known}
Let $G$ be an abelian group of order~$n$ and~$A \subseteq G \setminus \{ 0 \}$ with $|A| =k$.
Then~$A$ is sequenceable in the following cases:
\begin{enumerate}
\item $k \leq 9$~\cite{AL20};
\item $k\leq 12$ when $G$ is cyclic and $n=pt$, where $p$ is prime and $t\leq 4$~(see \cite{HOS19} and \cite{CDOR});
\item $k\leq 12$ when $G$ is cyclic and $n=mt$, where all the prime factors of $m$ are bigger than $k!/2$ and $t\leq 4$~\cite{CDOR};
\item $k = n-3$ when $n$ is prime and $\Sigma A \neq 0$~\cite{HOS19};
\item $k = n-2$ when $G$ is cyclic and $\Sigma A \neq 0$~\cite{BH05};
\item $k = n-1$~\cite{AKP17,Gordon61};
\item $n\leq 21$ and $n\leq 23$ when $\Sigma A=0$~\cite{CMPP18};
\item $n\leq 25$ when G is cyclic and $\Sigma A=0$ ~\cite{ADMS16}.
\end{enumerate}
\end{thm}
Other known results on this conjecture and related problems can be found in \cite{ADMS16, CDOR} and \cite{CP20}.

Inspired by the graph-theoretical interpretation, we propose the following weakening of the concept of sequenceability.
\begin{defn}
\begin{itemize}
\item Given a positive integer $t$ and a set $A$ whose cardinality is $k>t$, the ordering ${\bm \omega}$ is said to be a {\em $t$-weak sequencing} of~$A$ if the elements of ${\bm s}= (s_0, s_1, \ldots, s_k)$ are such that $s_i\not=s_j$ whenever $i\not=j$ and $|i-j|\leq t$;
\item a subset $A$ of an abelian group is said to be {\em $t$-weak sequenceable} if it admits a $t$-weak sequencing;
\item an abelian group $G$ is said to be {\em $t$-weak sequenceable} if every subset $A$ of $G\setminus \{0\}$ whose size is bigger than $t$ is $t$-weak sequenceable.
\end{itemize}
\end{defn}
Indeed, if a set $A$ admits a $t$-weak sequencing and $\Sigma A\not=0$, then the partial sums $(s_0,s_1,\ldots,s_k)$ define a walk in  $Cay[G:\pm A]$ whose girth is strictly bigger than $t$.
Exploiting this interpretation, we can state the analogous of Conjecture \ref{conj:main} for the weak sequenceability.
\begin{conj}\label{conj:mainWeak}
Let $t$ be a positive integer,~$G$ be an abelian group and let~$A$ be a finite subset of~$G \setminus \{ 0 \}$ whose size is $k>t$. Then~$A$ is $t$-weak sequenceable. Equivalently, we conjecture that any abelian group $G$ is $t$-weak sequenceable.
\end{conj}

It is worth also recalling another weaker version of the Alspach-Liversidge Conjecture, presented by Costa, Morini, Pasotti, and Pellegrini (see \cite{CMPP18}) that is sufficient for some applications to Heffter arrays (see also \cite{A15}):
\begin{conj}[CMPP conjecture]\label{conj:cmpp}
Let~$G$ be an abelian group and let~$A$ be a finite subset of~$G \setminus \{ 0 \}$ such that $|A \cap \{ x,-x \}| \leq 1$ for any~$x \in G$. Then~$A$ is sequenceable.
\end{conj}
Conjecture \ref{conj:cmpp} was presented in \cite{CMPP18} with the additional hypothesis that $\Sigma A=0$ but, since the authors of the present paper together with Anita Pasotti have recently proposed a variation of the Heffter arrays in which the rows and the columns do not sum to zero (see \cite{CostaDellaFiorePasotti} and also \cite{MellaPasotti}), we have stated here this slightly more general conjecture. Then, as done with Conjecture \ref{conj:main}, we propose a variation of the CMPP conjecture for the weak sequenceability.

\begin{conj}\label{conj:cmppWeak}
Let $t$ be a positive integer, ~$G$ be an abelian group and let~$A$ be a finite subset of~$G \setminus \{ 0 \}$ such that $|A \cap \{ x,-x \}| \leq 1$ for any~$x \in G$ and $|A|>t$. Then~$A$ is $t$-weak sequenceable.
\end{conj}

In this paper, we will work on these weak forms of Conjectures \ref{conj:main} and \ref{conj:cmpp}. 
In particular, when $G$ is the field $\mathbb{Z}_p$, we will use a polynomial approach whose starting point is the same as \cite{HOS19}. Here, after some manipulations, surprisingly, we will obtain a polynomial whose expression does not depend on the cardinality of $A$ and this allows us to have a result that is very general on the parameter $k=|A|$. On the other hand, since the degree of this polynomial increases very quickly in $t$, we can resolve, computationally,  using SageMath \cite{SageMath}, only the cases where $t$ is smaller respectively than $7$ and $8$ for Conjectures \ref{conj:mainWeak} and \ref{conj:cmppWeak}. These results will be presented in Section 2 of this paper. We also remark that the polynomial method provides better results than the ones that can be obtained by using the direct construction. 
Indeed, in the third section of this paper we will show how a direct approach, similar to that of \cite{CMPP18}, can effectively solve only very small values of $t$: we can prove, directly and with a great effort, Conjectures \ref{conj:mainWeak} and \ref{conj:cmppWeak} only when $t$ is smaller respectively than $4$ and $5$.

Finally, in the last section, we will outline a probabilistic approach. We will start from the result of \cite{ADMS16} that almost all the sets are sequenceable (fixed $|A|$ and asymptotically in $|G|$) and we will prove the existence of sequences that are not too far from being $t$-weak sequencings.
\section{Applying the Polynomial Method}\label{sec:poly}

In this section, we apply a method that relies on the Non-Vanishing Corollary of the Combinatorial Nullstellensatz, see~\cite{Alon99,Michalek10}.  Given a prime $p$ (in the following $p$ will be always assumed to be a prime), this corollary allows us to obtain a non-zero point to suitable polynomials on $\mathbb{Z}_p$ derived starting from the ones defined in \cite{HOS19}. Then, after some manipulations, surprisingly, we obtain a polynomial whose expression does not depend on the cardinality of $A$ and this allows us to have a result that is very general on the parameter $k=|A|$.

\begin{thm}\label{th:pm}{\rm (Non-Vanishing Corollary)}
Let~$\mathbb{F}$ be a finite field, and let $ f(x_1, x_2, \ldots, x_k)$ be a polynomial in~$\mathbb{F}[x_1, x_2, \ldots, x_k]$. Suppose the degree~$deg(f)$ of~$f$ is $\sum_{i=1}^k \gamma_i$, where each~$\gamma_i$ is a nonnegative integer, and suppose the coefficient of~$\prod_{i=1}^k x_i^{\gamma_i}$ in~$f$ is nonzero. If~$C_1, C_2, \ldots, C_k$ are subsets of~$\mathbb{F}$ with~$|C_i| > \gamma_i$, then there are $c_1 \in C_1, \ldots, c_k \in C_k$ such that~$f(c_1, c_2, \ldots, c_k) \neq 0$.
\end{thm}

In the notation of the Non-Vanishing Corollary,  we call the monomial $x_1^{|C_1| - 1} \cdots x_k^{|C_k| - 1}$ the {\em bounding monomial}. The corollary can be rephrased as requiring the polynomial to include a monomial of maximum degree that divides the bounding monomial (where by ``include" we mean that it has a nonzero coefficient).

To use the Non-Vanishing Corollary we require a polynomial for which the non-zeros correspond to successful solutions to the case of the problem under consideration.
We recall that, in order to attack Conjecture \ref{conj:main}, it was defined, in \cite{HOS19} the following polynomial
$$F_k(x_1,\cdots,x_k):=\prod_{1\leq i<j\leq k} (x_j-x_i) \prod_{\substack{0\leq i<j\leq k \\ j \neq i+1, (i,j) \neq (0,k)}} (x_{i+1}+\cdots+x_j).$$
It is clear that, given $A=\{a_1,\dots,a_k\}\subseteq \mathbb{Z}_p\setminus \{0\}$ of size $k$, $F_k(a_1,\dots,a_k)\not=0$ if and only if the sequence $(a_1,\dots, a_k)$ is a solution to Conjecture \ref{conj:main} for the set $A$.  In other words $A$ is sequenceable if and only if there exists an ordering that we denote, up to relabeling, with $(a_1,\dots, a_k)$ such that $F_k(a_1,\dots, a_k)\not=0$.

With respect to Conjecture \ref{conj:mainWeak}, in addition to requiring that $x_i-x_j\not=0$ for $1\leq i<j\leq k$, we seek an ordering to have no two of its partial sums $s_i, s_j$ equal for $1\leq i<j\leq k$ and $|i-j|\leq t$ (here there is the weakening of Conjecture \ref{conj:main}). Hence, modifying the expression of $F_k$, we define, for $t<k$, the following polynomial
$$P_{k,t}(x_1,\cdots,x_k):=\prod_{1\leq i<j\leq k} (x_j-x_i)\prod_{\substack{0\leq i<j\leq k \\ j-i\leq t, j \neq i+1}} (x_{i+1}+\cdots+x_j).$$
In this case we have that a set $A=\{a_1,\dots,a_k\}\subseteq \mathbb{Z}_p$ of size $k$ is $t$-weak sequenceable if and only if there exists an ordering $(a_1,\dots, a_k)$ of its elements such that $P_{k,t}(a_1,\dots, a_k)\not=0$.

Now, given a set $A=\{a_1,\dots,a_k\}$ of $k$ elements, the idea is to fix, {\em a priori}, the first $h$ elements $(a_1,\dots,a_h)$, where $h$ is not too big, of the ordering in such a way that no of its partial sums $s_i, s_j$ are equal for $1\leq i<j\leq h$ and $|i-j|\leq t$. This can be expressed by requiring that $P_{h,t}(a_1,\cdots,a_h)\not=0$ and we show that this can be done under the hypothesis of the following Proposition.
\begin{prop}\label{fix}
Let $A=\{a_1,\dots,a_k\}\subseteq \Z_p\setminus\{0\}$ be a set of size $k$ and let $h$ and $t$ be positive integers such that $h\leq k-(t-1)$.  Then there exists an ordering of $h$-elements of $A$ that we denote,  up to relabeling,  with $(a_1,\dots,a_h)$, such that
$$P_{h,t}(a_1,\cdots,a_h)\not=0.$$
\end{prop}
\begin{proof}
Given $k$ and $t$, we prove this statement by induction on $h$.

BASE CASE: Let $h=1$. Since $P_{1,t}(x)=1$ for any $t$ and for any $x\in A$, the statement is realized for $h=1$.

INDUCTIVE CASE: Let us assume the statement for $h\in \{1,\dots,m\}$ and let us prove it for $h=m+1$ where $m+1\leq k-(t-1)$.
Since the statement is true for $h=m$, there exists an $m$-tuple $(a_1,\dots,a_m)$ such that 
$$P_{m,t}(a_1,\cdots,a_m)\not=0.$$
We note that $$\frac{P_{m+1,t}(a_1,\dots,a_m,x)}{P_{m,t}(a_1,\dots,a_m)}=\prod_{1\leq i<m+1} (x-a_i)\prod_{\substack{0\leq i<m \\m+1-i\leq t}} (a_{i+1}+\cdots+a_m+x).$$
Here, any element $x$ of $A\setminus \{a_1,\dots,a_m\}$ satisfies
$\prod_{1\leq i<m+1} (x-a_i)\not=0$. Hence, to have $\frac{P_{m+1,t}(a_1,\dots,a_m,x)}{P_{m,t}(a_1,\dots,a_m)}\not=0$, it suffice to find $x$ such that $\prod_{\max(0,m+1-t)\leq i<m} (a_{i+1}+\cdots+a_m+x) \neq 0$.

Note that, for each relation $a_{i+1}+\cdots+a_m+x=0$ there is at most one solution $x\in A\setminus\{a_1,\dots,a_m\}$. Since those relations are at most $t-1$, we have at most $t-1$ values $x$ in $A\setminus\{a_1,\dots,a_m\}$ such that $\prod_{\max(0,m+1-t)\leq i<m} (a_{i+1}+\cdots+a_m+x)=0$. We recall that $m+1\leq k-(t-1)$ that is 
$$|A\setminus\{a_1,\dots,a_m\}|=k-m\geq t>t-1.$$
This means that there exists $a_{m+1}\in A\setminus\{a_1,\dots,a_m\}$ such that
$$\frac{P_{m+1,t}(a_1,\dots,a_m,a_{m+1})}{P_{m,t}(a_1,\dots,a_m)}=\prod_{1\leq i<m+1} (a_{m+1}-a_i)\prod_{\max(0,m+1-t)\leq i<m} (a_{i+1}+\cdots+a_{m+1})\not=0.$$
Since $\mathbb{Z}_p$ is a field and due to the inductive hypothesis $P_{m,t}(a_1,\dots,a_m) \neq 0$, we also have that
$$\frac{P_{m+1,t}(a_1,\dots,a_m,a_{m+1})}{P_{m,t}(a_1,\dots,a_m)}\cdot P_{m,t}(a_1,\dots,a_m)=P_{m+1,t}(a_1,\dots,a_{m+1})\not=0 $$
that completes the proof.
\end{proof}
In the following we assume that we have fixed, according to Proposition \ref{fix}, $\{a_1,\dots,$ $a_h\}\subseteq A$ such that $P_{h,t}(a_1,\dots,a_h)\not=0$. We note that every $x\in A\setminus\{a_1,\dots,a_h\}$ is such that $x-a_i\not=0$ for any $i\in \{1,\dots,h\}$. Therefore, it is left to find a nonzero point for the polynomial 
$$\frac{P_{k,t}(a_1,\dots,a_h,x_{h+1},\dots,x_{k})}{P_{h,t}(a_1,\dots,a_h)\prod_{1\leq i\leq h<j\leq k }(x_j-a_i)}.$$
Since the free variables are now $x_{h+1},\dots,x_k$, we set $\ell:=k-h$ and $y_i := x_{i+h}$; here the constrain $h\leq k-(t-1)$ of Proposition \ref{fix} becomes $\ell\geq t-1$. 
Then we denote by $H_{k,t,\ell}$ the polynomial
\begin{equation}\label{definizioneH}H_{k,t,\ell}(y_1,\dots,y_{\ell}):=\frac{P_{k,t}(a_1,\dots,a_{k-\ell},y_{1},\dots,y_{\ell})}{P_{k-\ell,t}(a_1,\dots,a_{k-\ell})\prod_{1\leq i\leq k-\ell;\ 1\leq j\leq \ell }(y_j-a_i)}.\end{equation} Assuming now that $k-\ell\geq t-1$, that is, $k-(t-1)\geq \ell\geq t-1$, we obtain the following expression
\begin{equation*}
H_{k,t,\ell}(y_1,\cdots,y_{\ell})=
\end{equation*}
\begin{equation}\label{espressioneH}
P_{\ell,t}(y_1,\dots,y_{\ell})\prod_{0\leq i\leq t-1;\ 1\leq j\leq t-i-1} (a_{k-\ell}+a_{k-\ell-1}+\cdots+a_{k-\ell-i}+y_1+y_2+\cdots+y_j).
\end{equation}
Now our aim is to apply the Non-Vanishing Corollary (i.e. Theorem \ref{th:pm}) to the polynomial $H_{k,t,\ell}$. At this purpose it is enough to consider the terms of $H_{k,t,\ell}$ of maximal degree in the variables $y_1,\dots,y_{\ell}$ that are the ones where no $a_i$ appears. We denote by $Q_{k,t,\ell}$ the polynomial given by those terms, that is
\begin{equation}\label{espressioneQ}Q_{k,t,\ell}:=P_{\ell,t}(y_1,\dots,y_{\ell})\prod_{\substack{0\leq i\leq t-1 \\ 1\leq j\leq t-i-1}} (y_1+y_2+\cdots+y_j).
\end{equation}
Now we can state the following, simple but  very powerful, remark.
\begin{rem}\label{MagicRemark}
The expression of $Q_{k,t,\ell}$ does not depend on $k$. In the following, we just denote this polynomial by $Q_{t,\ell}$.
\end{rem}
Indeed Remark \ref{MagicRemark} means that, after these manipulations, we are left to consider a polynomial that does not depend on $k=|A|$ and hence we have chances to get a result that is very general on $k$.

To apply the Non-Vanishing Corollary,  we also need to know the degree of $Q_{t, \ell}$ (and hence that of $H_{k, t, \ell}$) in order to compare it with the one of the bounding monomial.
\begin{lem}
$$ \deg(Q_{t,\ell})=(t-1)\ell+\frac{\ell(\ell-1)}{2}.$$
\end{lem}
\begin{proof}
It is more convenient to consider the degree of the polynomial $H_{k,t,\ell}$ defined in eq.~\eqref{espressioneH} since that of $Q_{t,\ell}$ is the same.

We note that, in $H_{k,t,\ell}$, each variable $y_j$ is the ending point of $t-1$ terms whose length is $i+1$ and that are of the form $(y_j+y_{j-1}+\dots+y_{j-i})$ or, if $i\geq j$, $(y_j+y_{j-1}+\dots+y_1+a_{k-\ell}+\dots+a_{k-\ell-(i-j)})$.

The other terms of $H_{k,t,\ell}$ are the one of the form $(y_j-y_i)$ where $\ell\geq j> i\geq 1$. Therefore we have that
$$\deg(H_{k, t,\ell})=(t-1)\ell+\frac{\ell(\ell-1)}{2}.$$
\end{proof}
To apply Theorem \ref{th:pm} we would need to find a nonzero coefficient of some monomials of type $\prod_{i=1}^{\ell} y_i^{\gamma_i}$ in $Q_{t,\ell}$ where each $\gamma_i$ is smaller than the number of choices for $y_i$, i.e. $\gamma_i\leq\ell-1$. Note that this is possible only if $\deg(Q_{t,\ell})\leq\ell(\ell-1)$: indeed this is the degree of the bounding monomial. Therefore we need that
$$(t-1)\ell+\frac{\ell(\ell-1)}{2}\leq \ell(\ell-1)$$
that is $2t-1\leq \ell$. Recalling that we are assuming $k-(t-1)\geq \ell\geq t-1$, the conditions on $\ell$ to apply Theorem \ref{th:pm} to the polynomial $H_{k,t,\ell}$ defined in eq. \eqref{espressioneH} are that
\begin{equation}\label{conditionELL}k-(t-1)\geq \ell\geq 2t-1.\end{equation}
From the previous discussion, it follows that
\begin{prop}\label{th:conditionWeakSeq}
Let $t,\ell$ be positive integers such that $\ell\geq 2t-1$. Let us also suppose that the coefficient of $\Pi_{i=1}^\ell y_i^{\gamma_i}$ in $Q_{t,\ell}$ is nonzero in $\mathbb{Z}_p$ where $\gamma_i\leq \ell-1$ for every $i\in \{1,\dots,\ell\}$. 

Then, any subset $A$ of $\mathbb{Z}_p\setminus \{0\}$ whose size is $k\geq \ell+(t-1)$ is $t$-weak sequenceable.
\end{prop}
\begin{rem}\label{CasiPiccoli} Since to apply Proposition \ref{fix} we need $\ell\geq t-1$, it is not guaranteed that we can find $\ell$ such that $k-\ell\geq t-1$. If this condition it is not satisfied, we can still define the polynomial $H_{k,t,\ell}$ via equation \eqref{definizioneH} even though equations \eqref{espressioneH} and \eqref{espressioneQ} do not hold. Then, since given $t$  those cases are only a finite number, we can apply directly the Non-Vanishing Corollary to equation \eqref{definizioneH}.
\end{rem}
Now we are ready to prove the main result of this paper
\begin{thm}\label{th:WeakSeqThm}
Let $t\leq 6$ be a positive integer, then for any prime $p$ the field $\mathbb{Z}_p$ is $t$-weak sequenceable.
\end{thm}
\begin{proof}
Since a $t$-weak sequenceable group is also $(t-1)$-weak sequenceable, we can suppose that $t = 6$.  By \cite{CDOR} we know that each subset $A \subseteq \mathbb{Z}_p \setminus \{0\}$ of size $k \leq 12$ is sequenceable,  therefore we can suppose that $k \geq 13$.  We divide the proof considering two different ranges of $k$.  

For each $13 \leq k \leq 15$ and for $k=16$ when $p$ is coprime with $379 \cdot 167938950753577$, the polynomial $H_{k, t, \ell}$, defined in eq. \eqref{definizioneH}, for ${\ell \in \{2t-1, 2t\}}$ has monomials with non-zero coefficients that divide the bounding monomial $y_1 ^{\ell-1} y_2 ^{\ell-1} \cdots y_{\ell}^{\ell-1}$, see Table \ref{tab:1}.  Then thanks to the Non-Vanishing Corollary each subset  $A \subseteq \mathbb{Z}_p$,  $|A| \leq 15$ (or $16$ under the above assumption), is sequenceable.

For $k \geq 17$ and for $k=16$ when $p$ is coprime with $3^4 \cdot 5 \cdot 47 \cdot 97 \cdot 271 \cdot 15985681$,  we consider the polynomials $Q_{t, \ell}$ defined in eq. \eqref{espressioneQ} for $\ell \in \{ 11, 12 \}$.  Since these polynomials have monomials that divide the bounding monomial with non-zero coefficients (see Table \ref{tab:2}),  and since $k \geq \ell + t - 1$,  we can apply Proposition \ref{th:conditionWeakSeq} to obtain that each subset $A \subseteq \mathbb{Z}_p$, $|A| \geq 17$ (or $16$ under the above assumption), is sequenceable. 

\begin{longtable}{lllll}
\caption{Monomials and their coefficients sufficient for the proof of Theorem~\ref{th:WeakSeqThm} in the case $A \subseteq \mathbb{Z}_p$,  $|A| \leq 16$.}\label{tab:1}\\
	\hline
	$k$ & $\ell$ & deg & monomial/s & coefficient/s \\
	\hline
	$16$ & $12$ & $125$ & $y_1^{5} y_2^{10} y_3^{11} y_4^{11} y_5^{11} y_6^{11} y_7^{11} y_8^{11} y_9^{11} y_ {11}^{11} y_{11}^{11} y_{12}^{11}$ &  $-379 \cdot 167938950753577$ \\
	\hline
	$15$ & $11$ & $109$ & \begin{tabular}{@{}l@{}} $y_1^9 y_2^{10} y_3^{10} y_4^{10} y_5^{10} y_6^{10} y_7^{10} y_8^{10} y_9^{10} y_ {10}^{10} y_{11}^{10}$ \\  $y_1^{10} y_2^{9} y_3^{10} y_4^{10} y_5^{10} y_6^{10} y_7^{10} y_8^{10} y_9^{10} y_ {10}^{10} y_{11}^{10}$ \end{tabular} &  \begin{tabular}{@{}l@{}}  $-3^4 \cdot 5 \cdot 47 \cdot 97 \cdot 271 \cdot 15985681$ \\ $-2^2 \cdot 3 \cdot 401 \cdot 1305987719053$ \end{tabular} \\
	\hline
	$14$ & $11$ & $107$ & \begin{tabular}{@{}l@{}} $y_1^7 y_2^{10} y_3^{10} y_4^{10} y_5^{10} y_6^{10} y_7^{10} y_8^{10} y_9^{10} y_ {10}^{10} y_{11}^{10}$ \\  $y_1^{8} y_2^{9} y_3^{10} y_4^{10} y_5^{10} y_6^{10} y_7^{10} y_8^{10} y_9^{10} y_ {10}^{10} y_{11}^{10}$ \end{tabular} &  \begin{tabular}{@{}l@{}}  $-2^2 \cdot 3 \cdot 5 \cdot 7^2 \cdot 37 \cdot 433 \cdot 81945547$ \\ $-3 \cdot 5 \cdot 555349 \cdot 496867859$ \end{tabular} \\
	\hline
	$13$ & $11$ & $104$ & \begin{tabular}{@{}l@{}} $y_1^5 y_2^9 y_3^{10} y_4^{10} y_5^{10} y_6^{10} y_7^{10} y_8^{10} y_9^{10} y_ {10}^{10} y_{11}^{10}$ \\  $y_1^{6} y_2^{8} y_3^{10} y_4^{10} y_5^{10} y_6^{10} y_7^{10} y_8^{10} y_9^{10} y_ {10}^{10} y_{11}^{10}$ \end{tabular} &  \begin{tabular}{@{}l@{}}  $-2 \cdot 11 \cdot 946021 \cdot 34341337$ \\ $-7 \cdot 211 \cdot 73019 \cdot 7962769$ \end{tabular} \\
 	\hline
\end{longtable}

\begin{longtable}{llll}
\caption{Monomials and their coefficients sufficient for the proof of Theorem~\ref{th:WeakSeqThm} in the case $A \subseteq \mathbb{Z}_p$,  $|A| \geq 16$}\label{tab:2}\\
	\hline
	$\ell$ & deg & monomial/s & coefficient/s \\
	\hline
	\endhead
	$11$ & $110$ & $y_1^{10} y_2^{10} y_3^{10} y_4^{10} y_5^{10} y_6^{10} y_7^{10} y_8^{10} y_9^{10} y_ {10}^{10} y_{11}^{10}$ & $-3^4 \cdot 5 \cdot 47 \cdot 97 \cdot 271 \cdot 15985681$\\
	\hline
	$12$ & $126$ & $y_1^{6} y_2^{10} y_3^{11} y_4^{11} y_5^{11} y_6^{11} y_7^{11} y_8^{11} y_9^{11} y_ {11}^{11} y_{11}^{11} y_{12}^{11}$ &  $-379 \cdot 167938950753577$\\
	\hline
\end{longtable}
\end{proof}
\subsection{Polynomial method for Conjecture \ref{conj:cmppWeak}}
We recall that, in order to attack Conjecture \ref{conj:cmpp}, it was defined, in \cite{HOS19} the following polynomial
$$\overline{F}_k(x_1,\cdots,x_k):=\frac{F_k(x_1,\cdots,x_k)}{\prod_{1\leq i< k} (x_i+x_{i+1})}.$$
Indeed, given $A=\{a_1,\dots,a_k\}\subseteq \mathbb{Z}_p\setminus \{0\}$ of size $k$ and such that $|A \cap \{ x,-x \}| \leq 1$ for any~$x \in  \mathbb{Z}_p$, we do not need to impose that $x_i+x_{i+1}$ is different from zero.
Therefore $\overline{F}_k(a_1,\dots,a_k)\not=0$ if and only if the sequence $(a_1,\dots, a_k)$ is a solution to Conjecture \ref{conj:cmpp} for the set $A$. Under the above assumptions, we can also say that $A$ is sequenceable if and only if there exists an ordering that we denote, up to relabeling, with $(a_1,\dots, a_k)$ such that $\overline{F}_k(a_1,\dots, a_k)\not=0$.

Reasoning in a similar way, with respect to Conjecture \ref{conj:cmppWeak}, we define, for $t<k$, the following polynomial
$$\overline{P}_{k,t}(x_1,\cdots,x_k):=\frac{P_{k,t}(x_1,\cdots,x_k)}{\prod_{1\leq i< k} (x_i+x_{i+1})}.$$
In this case we have that a set $A=\{a_1,\dots,a_k\}\subseteq \mathbb{Z}_p$ of size $k$ and such that $|A \cap \{ x,-x \}| \leq 1$ for any~$x \in \mathbb{Z}_p$ is $t$-weak sequenceable if and only if there exists an ordering $(a_1,\dots, a_k)$ of its elements such that $\overline{P}_{k,t}(a_1,\dots, a_k)\not=0$.

Then, with the same proof of Proposition \ref{fix}, and keeping in mind that we do not need to impose that $x_i+x_{i+1}\not=0$, we obtain that:
\begin{prop}\label{fix2}
Let $A=\{a_1,\dots,a_k\}\subseteq \Z_p\setminus\{0\}$ be a set of size $k$ such that $|A \cap \{ x,-x \}| \leq 1$ for any~$x \in \mathbb{Z}_p$, and let $h$ and $t$ be positive integers.  Then, if $h\leq k-(t-2)$,  there exists an ordering of $h$-elements of $A$ that we denote,  up to relabeling,  with $(a_1,\dots,a_h)$, such that
$$\overline{P}_{h,t}(a_1,\cdots,a_h)\not=0.$$
\end{prop}
Here we can assume that we have fixed, according to Proposition \ref{fix2}, $\{a_1,\dots,a_h\}\subseteq A$ such that $\overline{P}_{h,t}(a_1,\dots,a_h)\not=0$. Then, proceeding as we did with Conjecture \ref{conj:mainWeak}, we have that if we set $\ell=k-h$ and $\ell\geq t-2$ it is enough to find $y_1,\dots,y_{\ell}$ in $A\setminus \{a_1,\dots,a_{k-\ell}\}$ such that $\overline{H}_{k,t,\ell}(y_1,\dots,y_{\ell})\not=0$ where
\begin{equation}\label{definizioneH'}\overline{H}_{k,t,\ell}(y_1,\dots,y_{\ell}):=\frac{H_{k,t,\ell}(y_{1},\dots,y_{\ell})}{ (y_1 + a_{k-\ell}) \prod_{1\leq i< k} (y_i+y_{i+1})}.\end{equation} 

Assuming now that $k-\ell\geq t-1$, that is, $k-(t-1)\geq \ell\geq t-2$, we obtain the following expression
\begin{equation*}
\overline{H}_{k,t,\ell}(y_1,\cdots,y_{\ell})=
\end{equation*}
\begin{equation}\label{espressioneH'}
\overline{P}_{\ell,t}(y_1,\dots,y_{\ell})\prod_{\substack{ 0\leq i\leq t-1;\ 1\leq j\leq t-i-1\\ i+j>1}} (a_{k-\ell}+a_{k-\ell-1}+\cdots+a_{k-\ell-i}+y_1+y_2+\cdots+y_j).
\end{equation}
Here our aim is to apply the Non-Vanishing Corollary (i.e. Theorem \ref{th:pm}) to the polynomial $\overline{H}_{k,t,\ell}$. At this purpose it is enough to consider the terms of $\overline{H}_{k,t,\ell}$ of maximal degree in the variables $y_1,\dots,y_{\ell}$ that are the ones where no $a_i$ appears. We denote by $\overline{Q}_{k,t,\ell}$ the polynomial given by those terms, that is
\begin{equation}\label{espressioneQ'}\overline{Q}_{k,t,\ell}:=\overline{P}_{\ell,t}(y_1,\dots,y_{\ell})\prod_{\substack{ 0\leq i\leq t-1;\ 1\leq j\leq t-i-1\\ i+j>1}} (y_1+y_2+\cdots+y_j).
\end{equation}
Also in this case we can state this, simple but very powerful, remark.
\begin{rem}\label{MagicRemark2}
The expression of $\overline{Q}_{k,t,\ell}$ does not depend on $k$. In the following, we just denote this polynomial by $\overline{Q}_{t,\ell}$.
\end{rem}
In this case the degree of $\overline{Q}_{t, \ell}$ (and that of $\overline{H}_{k,t,\ell}$) is 
$$ \deg(\overline{Q}_{t,\ell})=(t-2)\ell+\frac{\ell(\ell-1)}{2}.$$
This means that, in order to apply the Non-Vanishing Corollary (i.e. Theorem \ref{th:pm}) to the polynomial $\overline{H}_{k,t,\ell}$ defined in eq. \eqref{espressioneH'}, we need that
\begin{equation}\label{conditionELL'}k-(t-1)\geq \ell\geq 2t-3.\end{equation}
From the previous discussion, it follows that
\begin{prop}\label{th:conditionWeakSeq'}
Let $t,\ell$ be positive integers such that $\ell\geq 2t-3$. Let us also suppose that the coefficient of $\Pi_{i=1}^\ell y_i^{\gamma_i}$ in $\overline{Q}_{t,\ell}$ is nonzero in $\mathbb{Z}_p$ where $\gamma_i\leq \ell-1$ for every $i\in \{1,\dots,\ell\}$. 

Then, any subset $A$ of $\mathbb{Z}_p\setminus \{0\}$ such that $|A \cap \{ x,-x \}| \leq 1$ for every~$x \in  \mathbb{Z}_p$ and whose size is $k\geq \ell+(t-1)$ is $t$-weak sequenceable.
\end{prop}
Now we are ready to prove our main result about Conjecture \ref{conj:cmppWeak}
\begin{thm}\label{th:WeakSeqThm'}
Let $t \leq 7$ be a positive integer and $A$ be a finite subset of~$\mathbb{Z}_p \setminus \{ 0 \}$ such that $|A \cap \{ x,-x \}| \leq 1$ for any~$x \in \mathbb{Z}_p$ and $|A|>t$. Then~$A$ is $t$-weak sequenceable.
\end{thm}
\begin{proof}
Since a $t$-weak sequenceable group is also $(t-1)$-weak sequenceable, we can suppose that $t = 7$.  By \cite{CDOR} we know that each subset $A \subseteq \mathbb{Z}_p \setminus \{0\}$ of size $k \leq 12$ is sequenceable, therefore we can assume that $k \geq 13$ and thus $p>13$.  We divide the proof considering two different ranges of $k$.  

For each $13 \leq k \leq 16$ and for $k=17$ when $p$ is coprime with $2 \cdot 7 \cdot 13 \cdot 4679 \cdot 3953841444019$, the polynomial $\overline{H}_{k, t, \ell}$,  defined in eq. \eqref{definizioneH'}, for $\ell \in \{2t-3,  2t-2\}$ has monomials with non-zero coefficients that divide the bounding monomial $y_1 ^{\ell-1} y_2 ^{\ell-1} \cdots y_{\ell}^{\ell-1}$,  see Table \ref{tab:3}.  Then thanks to the Non-Vanishing Corollary each subset $A \subseteq \mathbb{Z}_p \setminus \{0\}$, $|A| \leq 16$ (or $17$ under the assumption above),  that satisfies the hypothesis of Theorem \ref{th:WeakSeqThm'}, is sequenceable.

For $k \geq 18$ and for $k=17$ when $p$ is coprime with $13 \cdot 67 \cdot 451441944254443$,  we consider the polynomials $\overline{Q}_{t, \ell}$ defined in eq. \eqref{espressioneQ'} for $\ell \in \{ 11, 12 \}$.  Since these polynomials have monomials that divide the bounding monomial with non-zero coefficients (see Table \ref{tab:4}), and since $k \geq \ell + t - 1$,  we can apply Proposition \ref{th:conditionWeakSeq'} to prove Theorem \ref{th:WeakSeqThm'} for $k \geq 18$ (or $17$ under the assumption above).

\begin{longtable}{lllll}
\caption{Monomials and their coefficients sufficient for the proof of Theorem~\ref{th:WeakSeqThm} in the case $A \subseteq \mathbb{Z}_p$,  $|A| \leq 17$.}\label{tab:3}\\
	\hline
	$k$ & $\ell$ & deg & monomial/s & coefficient/s \\
	\hline
	$17$ & $12$ & $125$ & $y_1^{5} y_2^{10} y_3^{11} y_4^{11} y_5^{11} y_6^{11} y_7^{11} y_8^{11} y_9^{11} y_ {11}^{11} y_{11}^{11} y_{12}^{11}$ &  $2 \cdot 7 \cdot 13 \cdot 4679 \cdot 3953841444019$\\
	\hline
	$16$ & $11$ & $109$ & \begin{tabular}{@{}l@{}} $y_1^9 y_2^{10} y_3^{10} y_4^{10} y_5^{10} y_6^{10} y_7^{10} y_8^{10} y_9^{10} y_ {10}^{10} y_{11}^{10}$ \\  $y_1^{10} y_2^{9} y_3^{10} y_4^{10} y_5^{10} y_6^{10} y_7^{10} y_8^{10} y_9^{10} y_ {10}^{10} y_{11}^{10}$ \end{tabular} & \begin{tabular}{@{}l@{}} $13 \cdot 67 \cdot 451441944254443$ \\ $3^2 \cdot 281 \cdot 1163 \cdot 112116705839$  \end{tabular}  \\
	\hline
	$15$ & $11$ & $107$ & \begin{tabular}{@{}l@{}} $y_1^7 y_2^{10} y_3^{10} y_4^{10} y_5^{10} y_6^{10} y_7^{10} y_8^{10} y_9^{10} y_ {10}^{10} y_{11}^{10}$ \\  $y_1^{8} y_2^{9} y_3^{10} y_4^{10} y_5^{10} y_6^{10} y_7^{10} y_8^{10} y_9^{10} y_ {10}^{10} y_{11}^{10}$ \end{tabular} &  \begin{tabular}{@{}l@{}} $2^2 \cdot 59 \cdot 708923 \cdot 1059330263$ \\ $7 \cdot 149 \cdot 239 \cdot 4073 \cdot 212718109$ \end{tabular}\\
	\hline
	$14$ & $11$ & $104$ & \begin{tabular}{@{}l@{}} $y_1^5 y_2^{9} y_3^{10} y_4^{10} y_5^{10} y_6^{10} y_7^{10} y_8^{10} y_9^{10} y_ {10}^{10} y_{11}^{10}$ \\  $y_1^{6} y_2^{8} y_3^{10} y_4^{10} y_5^{10} y_6^{10} y_7^{10} y_8^{10} y_9^{10} y_ {10}^{10} y_{11}^{10}$ \end{tabular} &  \begin{tabular}{@{}l@{}} $2^3 \cdot 41 \cdot 7682093 \cdot 13267117$ \\ $2^2 \cdot 16834339 \cdot 679071929$   \end{tabular}\\
	\hline
	$13$ & $11$ & $100$ & \begin{tabular}{@{}l@{}} $y_1^2 y_2^{8} y_3^{10} y_4^{10} y_5^{10} y_6^{10} y_7^{10} y_8^{10} y_9^{10} y_ {10}^{10} y_{11}^{10}$ \\  $y_1^{3} y_2^{7} y_3^{10} y_4^{10} y_5^{10} y_6^{10} y_7^{10} y_8^{10} y_9^{10} y_ {10}^{10} y_{11}^{10}$ \end{tabular} &  \begin{tabular}{@{}l@{}} $3^3 \cdot 708569 \cdot 33345973$  \\ $3 \cdot 19 \cdot 7829 \cdot 31223 \cdot 121843$ \end{tabular}\\
	\hline
\end{longtable}

\begin{longtable}{llll}
\caption{Monomials and their coefficients sufficient for the proof of Theorem~\ref{th:WeakSeqThm} in the case $A \subseteq \mathbb{Z}_p$,  $|A| \geq 17$}\label{tab:4}\\
	\hline
	$\ell$ & deg & monomial/s & coefficient/s \\
	\hline
	\endhead
	$11$ & $110$ & $y_1^{10} y_2^{10} y_3^{10} y_4^{10} y_5^{10} y_6^{10} y_7^{10} y_8^{10} y_9^{10} y_ {10}^{10} y_{11}^{10}$ & $13 \cdot 67 \cdot 451441944254443$\\
	\hline
	$12$ & $126$ & $y_1^{6} y_2^{10} y_3^{11} y_4^{11} y_5^{11} y_6^{11} y_7^{11} y_8^{11} y_9^{11} y_ {11}^{11} y_{11}^{11} y_{12}^{11}$ &  $2 \cdot 7 \cdot 13 \cdot 4679 \cdot 3953841444019$\\
	\hline
\end{longtable}

\end{proof}
\section{Direct Construction(s)}
In this section we want to attack with a direct construction Conjecture \ref{conj:mainWeak}.  Even though we are able to solve it only for $t=3$, for this value, we obtain a solution in any cyclic group. Then we outline the proof of a similar statement for Conjecture \ref{conj:cmppWeak}. Indeed, in a very similar way, it is possible to prove that the latter conjecture holds, in cyclic groups, for $t\leq 4$.

First of all, we note that, with the same proof of Proposition \ref{fix}, we obtain the following:
\begin{prop}\label{fix3}
Let $A=\{a_1,\dots,a_k\}\subseteq \Z_n\setminus\{0\}$ be a set of size $k$ and let $h$ and $t$ be positive integers such that $h\leq k-(t-1)$. Then there is an ordering of $h$-elements of $A$ that we denote,  up to relabeling,  with $(a_1,\dots,a_h)$,  such that for any ${0\leq i<j\leq \min(h,i+t)}$
$$s_i=a_1+a_2+\dots+a_i\not=a_1+a_2+\dots+a_j=s_j.$$
Moreover, if $n$ is even and $n/2\in A$, we can assume $a_1=n/2$.
\end{prop}
Here we can not apply directly Proposition \ref{fix} because $\Z_n$ is not necessary a field.  However, with this new proposition we obtain the main result of this section.
\begin{thm}\label{th:WeakSeqDir}
Let $t\leq 3$ be a positive integer.  Then,  for any positive integer $n$,  the cyclic group $\mathbb{Z}_n$ is $t$-weak sequenceable.
\end{thm}
\begin{proof}
As usual it is enough to consider only the case $t=3$ and let $A=\{a_1,\dots,a_k\}$ be a subset of $\Z_n\setminus\{0\}$. Here we can assume $k>9$ because, due to the result of \cite{AL20}, Conjecture \ref{conj:main} holds for sets of size at most $9$. According to Proposition \ref{fix3}, for $h=k-4$, there is an ordering of $h$ elements of $A$ that we denote, up to relabeling,  with $(a_1,\dots,a_h)$,  such that $s_i\not=s_{j}$ whenever $0\leq i<j\leq \min(h,i+t)$. Moreover, if $n$ is even and $n/2\in A$, we can assume $a_1=n/2$. 

Then,  we need to order the last four elements $a_{k-3},a_{k-2},a_{k-1},a_k$. We divide the proof in two cases.

CASE 1: $a_{k-3}+a_{k-2}+a_{k-1}=0$. Here we first assume that $-a_{k}\not\in \{a_{k-3},a_{k-2},a_{k-1}\}$. In this case we look for an ordering of type $$(a_1,\dots,a_{k-5},a_{k-4},y_1,y_2,a_{k},y_3)$$ where $y_1,y_2,y_3\in \{a_{k-3},a_{k-2},a_{k-1}\}$. Indeed there exists $y_1\in \{a_{k-3},a_{k-2},a_{k-1}\}$ such that
$$\begin{cases}
y_1+a_{k-4}\not=0;\\
y_1+a_{k-4}+a_{k-5}\not=0.
\end{cases}$$
Then we can easily check that all the partial sums $s_i$ and $s_j$ such that $0\leq i<j\leq \min(k,i+t)$ are different.\\
Let us now assume that $-a_{k}=a_{k-3}$. Here if $a_{k}\not=a_{k-4}+a_{k-5}$ we can choose an ordering of type $$(a_1,\dots,a_{k-5},a_{k-4},-a_k,y_1,a_k,y_2)$$ where $y_1,y_2\in \{a_{k-2},a_{k-1}\}$. Indeed there exists $y_1\in \{a_{k-2},a_{k-1}\}$ such that
$$y_1-a_k+a_{k-4}\not=0.$$
Then we can easily check that all the partial sums $s_i$ and $s_j$ such that $0\leq i<j\leq \min(k,i+t)$ are different.\\
Let us now assume that $-a_{k}=a_{k-3}$ and $a_{k}=a_{k-4}+a_{k-5}$.
Here we have that $$\{a_{k-3},a_{k-2},a_{k-1},a_k\}=\{-(a_{k-4}+a_{k-5}),a_{k-2},a_{k-1},a_{k-4}+a_{k-5}\}$$ and we look for an ordering of type $$(a_1,\dots,a_{k-5},a_{k-4},y_1,a_{k-4}+a_{k-5},y_2,-(a_{k-4}+a_{k-5}))$$ where $y_1,y_2\in \{a_{k-2},a_{k-1}\}$. We first note that we can choose $y_1\in \{a_{k-2},a_{k-1}\}$ such that
$$\begin{cases}y_1+a_{k-4}\not=0;\\
y_1+2a_{k-4}+a_{k-5}\not=0.
\end{cases}$$
Indeed, if such $y_1$ does not exist, we would have that $$\{a_{k-3},a_{k-2},a_{k-1},a_k\}=\{-(a_{k-4}+a_{k-5}),-a_{k-4},-2a_{k-4}-a_{k-5},a_{k-4}+a_{k-5}\}.$$ We recall that also $a_{k-3}+a_{k-2}+a_{k-1}=0$ and hence it would follows that $-4a_{k-4}-2a_{k-5}=0$.  This means that $-2a_{k-4}-a_{k-5}$ is an involution and hence, because $\Z_n$ is cyclic, $-2a_{k-4}-a_{k-5}\in \{0,n/2\}$. Since $-2a_{k-4}-a_{k-5}\in \{a_{k-2},a_{k-1}\}$ this is in contradiction with the choice of the first element of the ordering. Then we can easily check that all the partial sums $s_i$ and $s_j$ such that $0\leq i<j\leq \min(k,i+t)$ are different.

CASE 2: Since we have already considered CASE 1, here we can assume, without loss of generality, that all the triples of elements do not sum to zero.  Moreover, let us assume that ${a_{k-3}+a_{k-2}=0}$. 
Here we can choose an ordering of type $$(a_1,\dots,a_{k-5},a_{k-4},y_1,z_1,y_2,z_2)$$ where $y_1,y_2\in \{a_{k-3},a_{k-2}\}$ and $z_1,z_2\in \{a_{k-1},a_{k}\}$. Indeed there exist $y_1\in \{a_{k-3},a_{k-2}\}$ and $z_1 \in \{a_{k-1},a_{k}\}$ such that
$$\begin{cases}
y_1+a_{k-4}+a_{k-5}\not=0;\\
z_1+y_1+a_{k-4}\not=0.
\end{cases}$$
Then we can easily check that all the partial sums $s_i$ and $s_j$ such that $0\leq i<j\leq \min(k,i+t)$ are different.

Finally, let us assume that all the triples and all the pairs of elements from $\{a_{k-3}, a_{k-2},$ $a_{k-1}, a_k\}$ do not sum to zero. Here, due to the pigeonhole principle, we can choose an ordering of type $$(a_1,\dots,a_{k-5},a_{k-4},y_1,y_2,y_3,y_3)$$ where $y_1,y_2$ are such that
$$\begin{cases}y_1+a_{k-4}\not=0;\\
y_1+a_{k-4}+a_{k-5}\not=0;\\
y_2+y_1+a_{k-4}\not=0.
\end{cases}$$
Then we can easily check that all the partial sums $s_i$ and $s_j$ such that $0\leq i<j\leq \min(k,i+t)$ are different.
\end{proof}
Moreover, with a very similar but more tedious proof, one could get an analogous result about Conjecture \ref{conj:cmppWeak}.
\begin{thm}\label{th:WeakSeqDir''}
Let $t \leq 4$ be a positive integer and $A$ be a finite subset of~$\mathbb{Z}_n \setminus \{ 0 \}$ such that $|A \cap \{ x,-x \}| \leq 1$ for any~$x \in \mathbb{Z}_p$ and $|A|>t$. Then~$A$ is $t$-weak sequenceable.
\end{thm}
This Theorem can be proved by fixing the first $k-6$ elements of the ordering and then choosing $a_{k-5}$ in such a way that $2a_{k-5}+2a_{k-6}+a_{k-7}\not=0$ and $s_{k-5}$ is different from $s_{k-8}$ and $s_{k-9}$. Then we can proceed, similarly to what we did for Theorem \ref{th:WeakSeqDir}, by considering two cases according to whether $a_{k-4}+a_{k-3}+a_{k-2}+a_{k-1}=0$ or not.
However, we prefer not to write a complete proof of this statement since we believe it goes beyond the scope of this paper and we have the feeling that it is not very deep from the mathematical point of view.
\section{A Probabilistic Approach}
In this section, we prove that a randomly chosen subset $A \subseteq \mathbb{Z}_n \setminus \{0\}$ of size $k$ is $t$-weak sequenceable when $t k$ is small with respect to $n$ and $t \geq 2$. This result is the corresponding of Theorem 4.2 of \cite{ADMS16} in the case of weak-sequenceability. In addition we also prove that for every $A \subseteq \mathbb{Z}_n \setminus \{0\}$ of size $k$ there exists an ordering $\bm \omega$ of $A$ with at most $t-2$ pairs of partial sums $s_i$,  $s_j$ such that $s_i = s_j$ and $|j-i| \leq t$.  

\begin{defn}
Let $\mathcal{A}_{n,k}$ be the set of all subsets of size $k$ of $\mathbb{Z}_n \setminus \{0\}$ that are $t$-weak sequenceable.  We say that {\em almost all $k$-subsets} of $\mathbb{Z}_n$ are $t$-weak sequenceable if
$$
	\lim_{n \to \infty} \frac{|\mathcal{A}_{n,k}|}{\binom{n-1}{k}} = 1\,.
$$
\end{defn}

\begin{prop}\label{th:RandomOrdSequence}
Let us choose an ordered sequence $ \bm \omega = (a_1,  a_2, \ldots, a_k)$ of distinct elements of $\mathbb{Z}_n \setminus \{0\}$ uniformly at random.  The probability that $\bm \omega$ is a $t$-weak sequencing of the set $\{a_1, a_2, \ldots, a_k\}$ is greater than or equal to $1 - \frac{(t-1)(k-2)}{n-2}$.
\end{prop}
\begin{proof}
Let $q$ be the probability that $\bm \omega$ is not a $t$-weak sequencing of $\{a_1, \ldots, a_k\}$. Then,  denoted by $\bm s = (s_0, s_1, \ldots, s_k)$ the partial sums of $\bm \omega$, we get
\begin{align*}
	q &\leq \sum_{\substack{0 \leq i < j\leq k \\ j - i \leq t, j \neq i+1}} \mathbb{P}(s_j = s_i) = \sum_{\substack{0 \leq i < j\leq k \\ j - i \leq t, j \neq i+1}} \mathbb{P}(a_{i+1} + \ldots + a_j = 0) \\
	&=  \sum_{\substack{0 \leq i < j\leq k \\ j - i \leq t, j \neq i+1}} \sum_{s = 0}^{n-1} \mathbb{P}(a_{i+1} + \ldots + a_{j-1} = s) \cdot \mathbb{P}(a_j = -s | a_{i+1} + \ldots + a_{j-1} = s) \,.
\end{align*}
We can upper bound $\mathbb{P}(a_j = -s | a_{i+1} + \ldots + a_{j-1} = s)$ by $1 / (n - j + i)$ for each $s \in \mathbb{Z}_n$ since there is at most one possible outcome for $a_j \in \mathbb{Z}_n \setminus \{0,  a_{i+1}, \ldots, a_{j-1}\}$, that makes the sum $a_{i+1} + \ldots + a_j = 0$. Hence
\begin{align*}
	q &\leq \sum_{\substack{0 \leq i < j\leq k \\ j - i \leq t, j \neq i+1}} \frac{1}{n - j + i} = \sum_{l = 2}^{t} \sum_{\substack{0 \leq i < j \leq k \\ j - i  = l}} \frac{1}{n-j+i} \\
	&= \sum_{l = 2}^{t} \frac{k-l}{n-l} \leq \frac{(t-1) (k-2)}{n-2} \, , 
\end{align*}
where we used the fact that $(k-l)/(n-l) \leq (k-2)/(n-2)$ for every $n > k \geq 2$ and $l \geq 2$. 
\end{proof}

\begin{thm}
Almost all $k$-subsets of $\mathbb{Z}_n$ are $t$-weak sequenceable when $tk = o(n)$ for $n \to \infty$.
\end{thm}
\begin{proof}
By Proposition \ref{th:RandomOrdSequence} it is easy to see that the probability a randomly chosen subset of $\mathbb{Z}_n \setminus \{0\}$ of size $k$ is not $t$-weak sequenceable is at most $\frac{(t-1) (k-2)}{n-2}$. 
\end{proof}

\begin{prop}\label{th:ExpValueOfEqualSums}
Let $A$ be a subset of size $k$ of $\mathbb{Z}_n \setminus \{0\}$ and let us choose uniformly at random an ordering $\bm \omega = (a_1, a_2, \ldots, a_k)$ of $A$.  Denoted the partial sums of $\bm \omega$ by $\bm s = (s_0, s_1, \ldots,  s_k)$,  let $X$ be the random variable that represents the number of pairs $(i, j)$ such that $s_i = s_j$ with $0\leq i<j \leq k$ and $j - i \leq t$, where $2 \leq t <k$.  

Then the expected value $\mathbb{E}(X)$ is smaller than $t-1$.
\end{prop}
\begin{proof}
By linearity of expectation we have that
$$
\mathbb{E}(X) = \sum_{\substack{0 \leq i < j\leq k \\ j - i \leq t, j \neq i+1}} \mathbb{P}(s_j = s_i)\,.
$$
Proceeding as in Proposition \ref{th:RandomOrdSequence} we get
\begin{align*}
 \mathbb{E}(X) &= \sum_{\substack{0 \leq i < j\leq k \\ j - i \leq t, j \neq i+1}} \mathbb{P}(a_{i+1} + \ldots + a_j = 0) \\
 & \leq \sum_{\substack{0 \leq i < j\leq k \\ j - i \leq t, j \neq i+1}}  \frac{1}{k - j + i + 1} = \sum_{l=2}^{t} \frac{k-l}{k-l+1} < t-1\,.
\end{align*}
\end{proof}

From Proposition \ref{th:ExpValueOfEqualSums} it immediately follows that

\begin{thm}
For every $A \subseteq \mathbb{Z}_n \setminus \{0\}$ and $2 \leq t < |A| $, there exists an ordering of $A$ with less than $t-1$ pairs of equal partial sums $s_i = s_j$ where $|j-i| \leq t$.
\end{thm}
\section*{Acknowledgements}
The first author was partially supported by INdAM--GNSAGA.

\end{document}